\documentclass[12pt,english,a4paper]{amsart}

\usepackage[all]{xy}
\usepackage{fullpage}
\usepackage{hyperref}

\newcommand{\Cc}{\mathbb{C}} 
\newcommand{\Pp}{\mathbb{P}}

\newcommand{\Zz}{\mathbb{Z}}
\newcommand{\Qq}{\mathbb{Q}} 
\newcommand\oline[1] {{\overline{#1}}}
\newcommand\ra{{\rightarrow}}

\DeclareMathOperator\Gal{Gal}

{\theoremstyle{plain}
\newtheorem{theorem}{Theorem}[section]    
\newtheorem{question}[theorem]{Question}
\newtheorem{lemma}[theorem]{Lemma}       
\newtheorem{proposition}[theorem]{Proposition}  
\newtheorem{corollary}[theorem]{Corollary}}

{\theoremstyle{remark}
\newtheorem{definition}[theorem]{Definition}      
\newtheorem{remark}[theorem]{Remark}
\newtheorem{example}[theorem]{Example}}

\begin{document}

\title{On parametric and generic polynomials \hskip 50mm with one parameter}

\author{Pierre D\`ebes}
\email{Pierre.Debes@math.univ-lille1.fr}

\author{Joachim K\"onig}
\email{jkoenig@knue.ac.kr}

\author{Fran\c cois Legrand}
\email{francois.legrand@tu-dresden.de}

\author{Danny Neftin}
\email{dneftin@technion.ac.il}

\address{Laboratoire de Math\'ematiques Paul Painlev\'e, Universit\'e de Lille, 59655 Villeneuve d'Ascq Cedex, France}

\address{Department of Mathematics Education, Korea National University of Education, 28173 Cheongju, South Korea}

\address{Institut f\"ur Algebra, Fachrichtung Mathematik, TU Dresden, 01062 Dresden, Germany}

\address{Department of Mathematics, Technion, Israel Institute of Technology, Haifa 32000, Israel}

\maketitle

\begin{abstract}
Given fields $k \subseteq L$, our results concern one para\-me\-ter {\it{$L$-parametric}} polynomials over $k$, and their relation to generic polynomials. The former are polynomials $P(T,Y) \in k[T][Y]$ of group $G$ which parametrize all Galois extensions of $L$ of group $G$ via specialization of $T$ in $L$, and the latter are those which are $L$-parametric for eve\-ry field $L \supseteq k$. We show, for example, that being $L$-parametric with $L$ taken to be the single field $\Cc((V))(U)$ is in fact sufficient for a polynomial $P(T, Y) \in \Cc[T][Y]$ to be generic. As a corollary, we obtain a complete list of one parameter generic polyno\-mials over a given field of characteristic 0, complemen\-ting the classical literature on the topic. Our approach also applies to an old problem of Schinzel: subject to the Birch and Swinnerton-Dyer conjecture, we provide one parameter families of affine curves over number fields, all with a rational point, but with no rational generic point.
\end{abstract}

\section{Introduction} \label{sec:intro}

Understanding the set ${\sf{R}}_G(k)$ of all finite Galois extensions of a given number field $k$ with given Galois group $G$ is a central objective in algebraic number theory. A first natural question in this area, which goes back to Hilbert and Noether, is the so-called {\it{inverse Galois problem}}: is ${\sf{R}}_G(k) \not= \emptyset$ for every number field $k$ and every finite group $G$? A more applicable goal is an explicit description of the sets ${\sf{R}}_G(k)$, such as a parametrization. A classical landmark in this context is the following definition (see the book \cite{JLY02}):

\begin{definition} \label{def:intro}
Let $k$ be a (number) field, $G$ a finite group, $n \geq 1$, and $P(T_1, \dots, T_n,Y) \in k[T_1, \dots, T_n][Y]$ a monic separable polynomial of group $G$ over $k(T_1, \dots, T_n)$.

\vspace{0.5mm}

\noindent
{\rm{(1)}} Given an overfield $L \supseteq k$, say that $P(T_1, \dots, T_n, Y)$ is {\it{$L$-parametric}} if, for every extension $E/L \in {\sf{R}}_G(L)$, there exists $(t_1, \dots, t_n) \in L^n$ such that $E$ is the splitting field over $L$ of $P(t_1, \dots, t_n,Y)$.

\vspace{0.5mm}

\noindent
{\rm{(2)}} Say that $P(T_1, \dots, T_n,Y)$ is {\it{generic}} if it is $L$-parametric for every overfield $L \supseteq k$.
\end{definition}

There are number fields $k$ and finite groups $G$ with no generic polynomial with coefficients in $k$, e.g., for $k=\Qq$ and $G=\Zz/8\Zz$  (see \cite[\S2.6]{JLY02}). However, if $k$ is a given number field and $G$ is a given finite group, it is in general unknown whether there is a generic polynomial of group $G$ with coefficients in $k$; existence of such polynomial would imply ${\sf{R}}_G(k) \not= \emptyset$, which is already open in general. Even for groups like $G=A_d$, for which ${\sf{R}}_G(k) \not= \emptyset$ is known, the question is open (for $d \geq 6$ and $k=\Qq$; see \cite[\S8.5]{JLY02}).

The case $n=1$ is better understood. If $k$ is any field of characteristic zero, finite groups $G$ with a generic polynomial $P(T,Y) \in k[T][Y]$ are precisely known (see \S\ref{ssec:basic_4}). In the case $k=\Qq$, those groups are exactly the subgroups of $S_3$. If $k$ is arbitrary, only cyclic groups and dihedral groups of order $2m$ with $m \geq 3$ odd can have a generic polynomial $P(T,Y) \in k[T][Y]$.

Here are the main stages of the classification.

\noindent
- If $G \not \subset {\rm{PGL}}_2(k)$ and if $P(T,Y) \in k[T][Y]$ is of group $G$, then $P(T,Y)$ is not generic; indeed, the Noether invariant extension $k({\bf{T}})/k({\bf{T}})^G$, with ${\bf T}= (T_1,\ldots,T_{|G|})$, cannot be reached by specializing $P(T,Y)$ at some $T=t_0 \in k({\bf{T}})^G$ (see \cite[Proposition 8.1.4]{JLY02}).  

\noindent
- If $G$ ($\subset {\rm{PGL}}_2(k)$ and) has a non-cyclic abelian subgroup, then the {\it{essential dimension}} theory of Buhler--Reichstein (see \cite{BR97, JLY02}) shows that there is no generic polynomial $P(T,Y) \in k[T][Y]$ of group $G$.

\noindent
- The remaining finite subgroups of ${\rm{PGL}}_2(k)$ of essential dimension 1 over $k$ are shown to have a generic polynomial $P(T,Y) \in k[T][Y]$.

We present a new approach, which allows more precise non-parametricity conclusions, over some specific fields, and leads to improvements on the above results. Theorem \ref{thm:intro_2} is a new general result on one parameter non-generic polynomials. Corollary \ref{cor:list} shows the concrete gain for the classification discussed above.

\begin{theorem} \label{thm:intro_2}
Let $k$ be a field of characteristic zero, $P(T,Y) \in k[T][Y]$ a monic separable polynomial, and $U, V$ two indeterminates. Suppose $P(T,Y)$ is not generic. Then either $P(T,Y)$ is not $\overline{k}((V))(U)$-parametric or $P(T,Y)$ is not $k(U)$-parametric. 
\end{theorem}

Thus Theorem \ref{thm:intro_2} gives explicit base changes $L/k$ such that any non-generic polynomial $P(T,Y) \in k[T][Y]$ is not $L$-parametric. Compared to the above approach, our base chan\-ges are purely transcendental, of transcendence degree 1, and do not depend on the Galois group $G$ of $P(T,Y)$ over $k(T)$, unlike the base change $k({\bf{T}})^G/k$ of the first stage above.

We refer to Corollary \ref{thm:intro_2bis} for a more general version of Theorem \ref{thm:intro_2}, which also presents some variants. For example, we prove that, if $G$ is neither cyclic nor dihedral of order $2n$ with $n \geq 3$ odd (in particular, $P(T,Y)$ is not generic), then $P(T,Y)$ is not $\overline{k}((V))(U)$-parametric.

The proof of Theorem \ref{thm:intro_2} uses a variety of tools, including the arithmetic specialization methods of \cite{KLN19}, patching methods from \cite{HHK11}, and the geometric specialization methods of \cite{DKLN18}.

Using Theorem \ref{thm:intro_2}, we obtain an explicit list of all the one parameter generic po\-ly\-no\-mials over any field $k$ of characteristic 0 (and not only the groups with such a polynomial). For simplicity, we give the list for $k=\Qq$ (see Corollary \ref{list ext} for the general case). 

\begin{corollary}\label{cor:list} 
Let $G$ be a non-trivial finite group and $P(T,Y) \in \Qq[T][Y]$ a monic separable polynomial of group $G$ and splitting field $F$ over $\Qq(T)$. Then $P(T,Y)$ is generic if and only if one of these conditions holds (up to a M\"obius transformation on $T$):

\vspace{1mm}

\noindent
{\rm{(1)}} $G=\Zz/2\Zz$ and $F$ is the splitting field over $\Qq(T)$ of $Y^2-T$,

\vspace{0.5mm}

\noindent
{\rm{(2)}} $G=\Zz/3\Zz$ and $F$ is the splitting field over $\Qq(T)$ of $Y^3- TY^2 + (T-3)Y + 1$,

\vspace{0.5mm}

\noindent
{\rm{(3)}} $G=S_3$ and $F$ is the splitting field over $\Qq(T)$ of $Y^3+TY+T$.
\end{corollary}

Note that the list is essentially known to experts, and the given polynomials are also known to be generic. The list complements the literature by showing that these indeed are the only one parameter generic polynomials (over $\Qq$). 

Our second type of results gives non-parametricity conclusions over $k$ itself (i.e., no base change $L/k$ is allowed). Such conclusions depend more on the arithmetic of $k$. For example, if $k$ is PAC, i.e., if every non-empty geometrically irreducible $k$-variety has a Zariski-dense set of $k$-rational points (see \cite{FJ08}), {\it{every}} polynomial $P(T,Y) \in k[T][Y]$ (whose splitting field $F$ over $k(T)$ fulfills $F \cap \overline{k}=k$) is $k$-parametric (see \cite{Deb99a})\footnote{There are PAC fields $k$ with ${\sf{R}}_G(k) \not= \emptyset$ for every finite group $G$, and so for which the $k$-parametricity property is not trivial as it is for algebraically closed fields. A concrete example (due to Pop) is $\mathbb{Q}^{\rm{tr}}(\sqrt{-1})$.}. 

However, if $k$ is a number field, it is expected that, as in the generic situation, most finite groups fail to have a $k$-parametric polynomial $P(T,Y) \in k[T][Y]$. Yet, no such group was known before \cite{KL18, KLN19}, which provide many examples. But the question remains of the classification of all the finite groups $G$ with a $k$-parametric polynomial $P(T,Y) \in k[T][Y]$, and of the associated polynomials. Could it be the same as in the generic situation? For $n \geq 2$, this is not the case: $\Zz/8\Zz$ has a $\Qq$-parametric polynomial $P(T_1, \dots, T_n, Y) \in \Qq[T_1, \dots, T_n][Y]$ for some $n \geq 2$ (see \cite{Sch92}) but no generic polynomial over $\Qq$. For $n=1$, the question is subtler. We show that, for polynomials $P(T,Y) \in k[T][Y]$, ``$k$-parametric" is still strictly weaker than ``generic", and even weaker than ``$k(U)$-parametric", but only under the Birch and Swinnerton-Dyer conjecture.

This comparison between the various notions relates to the following old problem of Schinzel (see \cite[Chapter 5, \S5.1]{Sch00}):

\begin{question} \label{pb:schinzel} 
Let $k$ be a number field and $P\in \Cc[U,T,Y]$ a polynomial such that, for all but finitely many $u_0\in \Zz$, the polynomial $P(u_0,T,Y)$ has a zero in $k^2$. Does P, viewed as a polynomial in T and Y, have a zero in $k(U)^2$?
\end{question}

We combine previous works producing ``lawful evil" elliptic curves (see, e.g., \cite{DD09, LY20}) and a result of \cite{Deb18} on the branch point number of rational pullbacks of Galois covers of $\mathbb{P}^1$ to obtain infinitely many (conditional) counter-examples to Question \ref{pb:schinzel}. For example, we have the following (see Theorem \ref{thm versus} for a more general result):

\begin{theorem} \label{thm:main4}
Let $Q(T)\in \mathbb{Q}[T]$ be a degree 3 separable polynomial such that the el\-lip\-tic curve given by $Y^2=Q(T)$ has complex multiplication by $\mathbb{Q}(\sqrt{-m})$ for some $m\in \{11, 19, 43, 67, 163\}$.  Set $P(U,T,Y) = Y^2 - UQ(T)$. Under the Birch and Swinnerton-Dyer conjecture, there exist infinitely many quadratic number fields $k$ such that 

\vspace{0.5mm}

\noindent
{\rm{(1)}} the answer to Ques\-tion \ref{pb:schinzel} is negative for $k$ and $P$, and

\vspace{0.5mm}

\noindent
{\rm{(2)}} the polynomial $P(1,T,Y)$ and the field $k$ form a counter-example to this implication:

\vspace{0.5mm}

\noindent
\hskip 40mm {\it {\rm{($*$)}} \, \, $k$-parametric $\Rightarrow$ $k(U)$-parametric.}

\vspace{0.5mm}

\noindent
\end{theorem}

A concrete situation where Theorem \ref{thm:main4} applies is $Q(T)=T^3 - T^2 - 7T + 41/4$, $m=11$, and $k = \Qq(\sqrt{-d})$, where $d >0$ is squarefree and fulfills $(\frac{d}{11}) = 1$ (see Remark \ref{rem:versus}).

All counter-examples to Question \ref{pb:schinzel} known to us are in genus $1$, i.e., when the $\Cc(U)$-curve $P(U,T, Y)=0$ is of genus 1 \footnote{For an earlier counter-example, subject to a conjecture of Selmer, see \cite[pp. 318-319]{Sch00}}. This suggests that the genus 1 case is exceptional. It remains plausible that the answer to Question \ref{pb:schinzel} is ``Yes" in genus $\geq 2$ \footnote{In genus 0, the answer to Question \ref{pb:schinzel} is ``Yes'' (see \cite[Theorem 2]{DLS65/66} and \cite[Theorem 38]{Sch82}).}. The same could be true of Implication ($*$). Indeed, by \cite{Deb18}, a positive answer to (a close variant) of Ques\-tion \ref{pb:schinzel} implies (a close variant of) Implication ($*$). See \S\ref{sec:wh} for more details.

Our results are stated in terms of polynomials $P(T,Y)$. They translate immediately in terms of field extensions $F/k(T)$. The latter viewpoint is the one that we will prefer in the sequel. We refer to Lemma \ref{prop:compare_1} for the exact connection between the two viewpoints.

\vspace{2mm} 

{\it Acknowledgments}. We thank Danny Krashen and the referee for helpful comments. This work was supported in part by the Labex CEMPI (ANR-11-LABX-0007-01), and ISF grants No. 577/15 and No. 696/13.

\section{Preliminaries} \label{sec:notation}

\S\ref{ssec:basic_1} and \S\ref{ssec:basic_2} are devoted to terminology and notation. In \S\ref{ssec:basic_4}, we review the classification of all the finite groups with a one parameter generic polynomial with coefficients in a given field of characteristic zero (see Theorem \ref{prop:classification}). As said in \S\ref{sec:intro}, Corollary \ref{list ext} will more precisely give the list of corresponding polynomials. Finally, in \S\ref{ssec:basic_5}, we briefly discuss \'etale algebras, which will be used to prove Theorem \ref{thm:intro_2} and its generalizations.

\subsection{Basic terminology} \label{ssec:basic_1}

Let $k$ be a field of characteristic zero and $F/k(T)$ a finite Galois extension. Say that $F/k(T)$ is {\it{$k$-regular}} if $F \cap \overline{k}=k$. By the \textit{{branch point set}} of $F/k(T)$, we mean the finite set ${\bf{t}}$ of points $t\in \Pp^1(\overline k)$ such that the associated discrete va-

\newpage

\noindent
luations are ramified in $F\overline{k}/\overline k(T)$. As $k$ has characteristic $0$, we have the \textit{{inertia canonical invariant}} ${\bf C}$ of $F/k(T)$: if ${\bf t}=\{t_1,\ldots,t_r\}$, then ${\bf C}$ is an $r$-tuple $(C_1,\ldots,C_r)$ of conjugacy classes of ${\rm{Gal}}(F\overline{k}/\overline{k}(T))$, and $C_i$ is the conjugacy class of the distinguished generators of the inertia groups $I_{\frak P}$ above $t_i$ in $F \overline{k}/\overline{k}(T)$ (i.e., these generators correspond to  $e^{2i\pi/e_i}$ in the canonical isomorphism  $I_{\frak P} \rightarrow \mu_{e_i} =\langle e^{2i\pi/e_i} \rangle$). We also use the notation ${\bf e}=(e_1,\ldots,e_r)$ for the $r$-tuple whose $i$-th entry is the ramification index $e_i=|I_{\frak P}|$ of primes above $t_i$; $e_i$ is also the order of elements of $C_i$. By the {\it{genus}} of $F/k(T)$, we mean the genus of $F \overline{k}$.

The {\it{specialization}} $F_{t_0}/k$ of $F/k(T)$ at $t_0 \in {\mathbb P}^1(k)$ is defined as follows. For $t_0\in k$ (resp., $t_0=\infty$), the field $F_{t_0}$ is the residue field of the integral closure of $k[T]$ (resp., of $k[1/T]$) in $F$ at any prime ideal $\mathfrak{P}$ containing $T-t_0$ (resp., $1/T$). The extension $F_{t_0}/k$ is Galois and, if $t_0 \not \in {\bf{t}}$, its Galois group is the decomposition group of $F/k(T)$ at $\mathfrak{P}$. If $F$ is the splitting field over $k(T)$ of a monic separable polynomial $P(T,Y) \in k[T][Y]$, then, for $t_0\in k$ with $P(t_0,Y)$ separable, $t_0 \not \in {\bf{t}}$ and $F_{t_0}$ is the splitting field over $k$ of $P(t_0,Y)$.

\begin{lemma} \label{spec}
We have $(FL)_{t_0}=F_{t_0}L$ for every overfield $L \supseteq k$ and every $t_0 \in \mathbb{P}^1(k)$.
\end{lemma}

\begin{proof}
Without loss, we may assume $t_0 \not=\infty$. Denote the integral closure of $k[T]$ in $F$ by $B_k$. Pick $s \geq 1$ and $b_1,\dots,b_{s}$ in $B_k$ with $B_k=k[T] b_1 + \dots + k[T]b_{s}.$ As $k$ has characteristic zero, $L/k$ is separable (in the sense of non-necessarily algebraic extensions; see, e.g., \cite[Chapter VIII, \S4]{Lan02}). Then, by, e.g., \cite[Proposition 3.4.2]{FJ08}, the integral closure $B_L$ of $L[T]$ in $FL$ equals $L[T]b_1 + \dots + L[T]b_{{s}}.$ Let $\frak{P}_L$ be a prime ideal of $B_L$ containing $T-t_0$. Then the restriction $\frak{P}_k=\frak{P}_L \cap B_k$ of $\frak{P}_L$ to $B_k$ also contains $T-t_0$. We then have $(F L)_{t_0}=B_L/\frak{P}_L = L(\overline{b_1},\dots, \overline{b_{s}}),$ where $\overline{b_1},\dots, \overline{b_{s}}$ denote the reductions modulo $\frak{P}_L$ of ${b_1},\dots, {b_{s}}$, respectively. But these reductions modulo $\frak{P}_L$ are the reductions $\underline{b_1}, \dots, \underline{b_{s}}$ modulo $\frak{P}_k$ of $b_1,\dots,b_{s}$, respectively. Hence, $(F L)_{t_0} = k(\underline{b_1}, \dots, \underline{b_{s}})L = F_{t_0}L$.
\end{proof}

\subsection{Generic and parametric extensions} \label{ssec:basic_2}

Given a finite group $G$ and a field $k$ (of characteristic zero), we will use the following notation.

\noindent
$\bullet$ ${\sf R}_G(k)$: set of all Galois extensions $E/k$ of group $G$.

\noindent
$\bullet$ ${\sf R}_{\leq G}(k)$: set of all Galois extensions $E/k$ of group contained in $G$.

\noindent
$\bullet$ For a finite Galois extension $F/k(T)$ of branch point set ${\bf{t}}$, we define
$${\sf SP}(F/k(T)) = \{ F_{t_0}/k \hskip 2pt | \hskip 1mm t_0\in \Pp^1(k) \setminus {\bf{t}} \} \, \, \, \, \, ({\sf{SP}} \,  \, {\rm{for}} \, \, ``{\rm{SPecialization}}").$$

\begin{definition} \label{def:main} 
Let $k$ be of characteristic 0 and $F/k(T)$ a Galois extension of group $G$.

\vskip 0.5mm

\noindent
(1) Given an overfield $L \supseteq k$, the extension $F/k(T)$ is {\it $L$-parametric} (resp., {\it{strongly $L$-parametric}}) if ${\sf SP}(FL/L(T)) \supseteq {\sf R}_{G}(L)$ (resp., if ${\sf SP}(FL/L(T)) = {\sf R}_{\leq G}(L)$).

\vskip 0.5mm

\noindent
(2) The extension $F/k(T)$ is ${\it{generic}}$ if it is $L$-parametric for every overfield $L \supseteq k$.
\end{definition} 

The field extension viewpoint used in Definition \ref{def:main} is of course equivalent to the polynomial one used in \S\ref{sec:intro}. The next two lemmas, which will be used on several occasions in the sequel, provide the precise arguments to pass from one viewpoint to the other.

\begin{lemma} \label{prop:compare_1}
Let $k$ be a field of characteristic zero, $G$ a finite group, $F/k(T)$ a Galois extension of group $G$, and $P(T,Y) \in k[T][Y]$ a monic separable polynomial of splitting field $F$ over $k(T)$. Let $L \supseteq k$ be any overfield.

\vspace{0.5mm}

\noindent
{\rm{(1)}} Let $E/L \in {\sf{R}}_{G}(L)$. If $E$ is the splitting field over $L$ of $P(t_0,Y)$ for some $t_0 \in L$, then $E/L \in {\sf{SP}}(FL/L(T))$.

\vspace{0.5mm}

\noindent
{\rm{(2)}} Let $E/L \in {\sf{R}}_{\leq G}(L)$. If there exist infinitely many $t_0 \in \mathbb{P}^1(L)$ such that $E/L = (FL)_{t_0}/L$, then $E$ is the splitting field over $L$ of $P(t_0,Y)$ for infinitely many $t_0 \in L$.
\end{lemma}

\begin{proof}
(1) Assume there is $t_0 \in L$ such that $E$ is the splitting field over $L$ of $P(t_0,Y)$. It is always true that the splitting field of $P(t_0,Y)$ is contained in the field $(FL)_{t_0}$. As, in our situation, the former is of degree $|G|$ over $L$, both fields coincide and we get $E=(FL)_{t_0}$. To conclude, note that $t_0$ is not a branch point of $F/k(T)$, since ${\rm{Gal}}((FL)_{t_0}/L)=G$.

\vspace{0.5mm}

\noindent
(2) By the assumption, there are infinitely many $t_0 \in \mathbb{P}^1(L)$ with $E=(FL)_{t_0}$. For such a $t_0$ with $t_0 \not = \infty$ and $P(t_0,Y)$ separable, the splitting field of $P(t_0,Y)$ over $L$ equals $(FL)_{t_0}$ (as recalled in \S\ref{ssec:basic_1}), i.e., equals $E$.
\end{proof}

Our second lemma adjusts \cite[Proposition 5.1.8]{JLY02} to our situation:

\begin{lemma} \label{prop:compare_2}
Let $k$ be a field of characteristic zero and $F/k(T)$ a generic extension. Then there exists a generic polynomial $P(T,Y) \in k[T][Y]$ of splitting field $F$ over $k(T)$.
\end{lemma}

\begin{proof}
First, denote the integral closure of $k[T]$ in $F$ by $B_k$. Pick a positive integer $s$ and an $s$-tuple $(b_1,\dots,b_s)$ of elements of $B_k$ with $B_k=k[T] b_1 + \dots + k[T]b_s$. Up to reordering, we may assume there exists $s' \leq s$ satisfying these two conditions:

\noindent
$\bullet$ for $1 \leq i \not= j\leq s'$, $b_i$ and $b_j$ are not conjugate over $k(T)$,

\noindent
$\bullet$ for $i > s'$, there exists $1 \leq j \leq s'$ such that $b_i$ and $b_j$ are conjugate over $k(T)$.

\noindent
For each $i \in \{1,\dots,s'\}$, denote the minimal polynomial of $b_i$ over $k(T)$ by $m_i(T,Y)$. Set $P_1(T,Y) = \prod_{i=1}^{s'} m_i(T,Y)$. Then $P_1(T,Y)$ is a monic separable polynomial with coefficients in $k[T]$, and its splitting field over $k(T)$ is equal to $F$.

Now, let $L \supseteq k$ and $E/L \in {\sf{R}}_{\leq G}(L)$ (with $G = {\rm{Gal}}(F/k(T))$). Assume $E/L=(FL)_{t_0}/L$ for some $t_0 \in L$. Let $\mathfrak{P}_L$ be a maximal ideal of the integral closure $B_L$ of $L[T]$ in $FL$ containing $T-t_0$. As in the proof of Lemma \ref{spec}, we have $B_L= L[T] b_1 + \dots + L[T] b_s$. Hence, with $\overline{b_1}, \dots, \overline{b_s}$ the respective reductions modulo $\mathfrak{P}_L$ of $b_1,\dots,b_s$, we have $(FL)_{t_0} = B_L/\mathfrak{P}_L = L(\overline{b_1},\dots,\overline{b_s})$. Thus, $E$ is the splitting field over $L$ of $P_1(t_0,Y)$.

Next, if we replace $T$ by $1/T$, the same arguments yield a monic separable polynomial $P_2(T,Y) \in k[T][Y]$ of splitting field $F$ over $k(T)$ which fulfills this: for all $L \supseteq k$ and $E/L \in {\sf{R}}_{\leq G}(L)$, if $E=(FL)_\infty$, then $E$ is the splitting field over $L$ of $P_2(0,Y)$.

Finally, since $F/k(T)$ is assumed to be generic, we get the following: for every overfield $L \supseteq k$ and every extension $E/L \in {\sf{R}}_G(L)$, there are $i \in \{1,2\}$ and $t_0 \in L$ such that $E$ is the splitting field over $L$ of $P_i(t_0,Y)$. It then remains to use \cite[Corollary 1.1.6]{JLY02} to get that either $P_1(T,Y)$ or $P_2(T,Y)$ is generic.
\end{proof}

\subsection{Finite groups with a one parameter generic polynomial/extension} \label{ssec:basic_4}

We give the classification of these groups over any given field of characteristic zero:

\begin{theorem} \label{prop:classification}
Let $k$ be a field of characteristic $0$ and $G$ a finite group. Then the following three conditions are equivalent:

\vspace{0.5mm}

\noindent
{\rm{(1)}} there exists a generic extension $F/k(T)$ of group $G$,

\vspace{0.5mm}

\noindent
{\rm{(2)}} there exists a generic polynomial $P(T,Y) \in k[T][Y]$ of group $G$,

\vspace{0.5mm}

\noindent
{\rm{(3)}} one of the following three conditions holds:

\vspace{0.5mm}

{\rm{(a)}} $G$ is cyclic of even order $n$ and $e^{2i \pi/n} \in k$,

\vspace{0.5mm}

{\rm{(b)}} $G$ is cyclic of odd order $n$ and $e^{2i \pi/n} + e^{-2i \pi/n} \in k$,

\vspace{0.5mm}

{\rm{(c)}} $G$ is dihedral of order $2n$ with $n \geq 3$ odd and $e^{2i \pi/n} + e^{-2i \pi/n} \in k$.
\end{theorem}

\begin{proof}
Implication (2) $\Rightarrow$ (1) is an immediate consequence of Lemma \ref{prop:compare_1}(1) while Lemma \ref{prop:compare_2} yields Implication (1) $\Rightarrow$ (2). It then remains to show (2) $\Leftrightarrow$ (3). This equivalence is known to experts but does not seem to appear explicitly in the literature. For the convenience of the reader, we recall the main ingredients.

First, assume there is a generic polynomial $P(T,Y) \in k[T][Y]$ of group $G$. By \cite[Proposition 8.2.4]{JLY02}, the {\it{essential dimension}} of $G$ over $\overline{k}$ is 1. Over fields of characte\-ris\-tic 0 containing all roots of unity, such groups are exactly cyclic groups and dihedral groups of order $2n$ with $n \geq 3$ odd (see \cite[Theorem 6.2]{BR97}). Hence, $G$ is cyclic or dihedral of order $2n$ with $n \geq 3$ odd.

Now, suppose $G=\Zz/n\Zz$ ($n \geq 2$). Using again \cite[Proposition 8.2.4]{JLY02}, if there is a generic polynomial $P(T,Y) \in k[T][Y]$ of group $G$, the essential dimension of $G$ over $k$ is 1. Theorem 1.3 of \cite{CHKZ08} then yields $e^{2i \pi/n} + e^{-2i \pi/n} \in k$ if $n$ is odd, and $e^{2i \pi/n} \in k$ if $n$ is even. Conversely, assume either $n$ is even and $e^{2i \pi/n} \in k$, or $n$ is odd and $e^{2i \pi/n} + e^{-2i \pi/n} \in k$. In the even case, $Y^n-T$ has Galois group $G$ over $k(T)$ and is generic (by the Kummer theory). Now, in the odd case, there is a generic polynomial $P(T,Y) \in k[T][Y]$ of group $G$, by \cite[\S2.1 and Exercise 5.13]{JLY02}. 

Finally, given $n\geq 3$ odd, suppose $G$ is dihedral of order $2n$. If there is a generic polynomial $P(T,Y) \in k[T][Y]$ of group $G$, then, by \cite[Proposition 8.2.4]{JLY02} and \cite[Theorem 1.4]{CHKZ08}, we have $e^{2i \pi/n} + e^{-2i \pi/n} \in k$. Conversely, if $e^{2i \pi/n} + e^{-2i \pi/n} \in k$, then, by a classical construction of Hashimoto and Miyake (see \cite[Theorem 5.5.4]{JLY02}), there is a generic polynomial $P(T,Y) \in k[T][Y]$ of group $G$.
\end{proof}

\begin{remark}
Groups with essential dimension 1 over a given field $k$ (of any charac\-teristic) are classified in \cite{CHKZ08}. It might then be possible to derive the list of all groups with a one parameter generic polynomial/extension over $k$, as done above in characteristic $0$. 
\end{remark}

\subsection{\'Etale algebras} \label{ssec:basic_5}

Let $k$ be a field of characteristic $0$. A general reference for this section is  \cite[\S4.3]{JLY02}. Every $G$-Galois extension $L/k$ of \'etale algebras is an {\it induced} from a Galois field extension $L'/k$ whose Galois group $H={\rm{Gal}}(L'/k)$ is a subgroup of $G$. In particular, $L=\bigoplus_{\sigma H\in G/H} \sigma(L')$. The Galois group of the field extension $L'/k$ is then its stabilizer under the action of $G$. The {\it underlying field} $L'$ is determined up to choosing a direct summand of $L$. When picking a conjugate copy $\sigma(L')$, the resulting stabilizer is the conjugate subgroup $\sigma H\sigma^{-1}$. 

\begin{remark}\label{rem:comp}
Given an overfield $M \supseteq k$, the tensor product $(L\otimes_k M)/M$ is well-known to be a $G$-Galois extension of \'etale algebras, whose underlying field is a compositum $L\cdot M$ via some embedding of $\oline k$ into $\oline M$. 
\end{remark}

The induced $G$-Galois extension $k^{|G|}$ from the trivial extension $k/k$ is called the {\it split} $G$-Galois extension. 

\section{Parametricity and Genericity} \label{sec:Parametricity} 

\S\ref{ssec:main} states the main results of this section. These results are proved in \S\ref{ssec:proof1_2}--\S\ref{ssec:proof2}. Fi\-nal\-ly, in \S\ref{ssec:proof_0}, we explain how Theorem \ref{thm:intro_2} and Corollary \ref{cor:list} are derived. 

\subsection{Main results} \label{ssec:main}

Let $G$ be a non-trivial finite group, $k$ a field of characteristic 0, and $U, V$ two indeterminates. Let $F/k(T)$ be a Galois extension of group $G$ and branch point set ${\bf{t}}=\{t_1,\dots,t_r\}$. We also denote the genus of $F$ by $g$ and the ramification indices of $t_1, \dots, t_r$ by $e_1, \dots, e_r$, respectively. The unordered $r$-tuple $(e_1,\dots,e_r)$ is denoted by {\textbf{e}}. 

The main topic of this section is this question: 

\vspace{0.75mm}

\noindent
{\it{$(*)$ Given an overfield $L \supseteq k$, is $F/k(T)$ $L$-parametric?}} 

\vspace{0.75mm}

\noindent
In the next result, we give three explicit base changes $L_1/k$, $L_2/k$, and $L_3/k$, independent of either the extension $F/k(T)$ or the group $G$, such that the answer to ($*$) is negative in general, if $L$ is taken among the fields $L_1$, $L_2$, and $L_3$.

Recall that a field $K$ is {\it{ample}} (or {\it{large}}) if every smooth $K$-curve has 0 or infinitely many $K$-rational points. Ample fields include algebraically closed fields, the complete va\-lued fields $\Qq_p$, $\mathbb{R}$, $\kappa((Y))$, the field $\Qq^{\rm{tr}}$ of totally real numbers (algebraic numbers such that all conjugates are real). See \cite{Jar11, BSF13, Pop14} for more details.

\begin{theorem} \label{list}
{\rm{(1)}} Let $K \supseteq k$ be an ample overfield and $L_1=K(U)$. Assume $g \geq 1$. Then ${\sf{R}}_G(L_1) \setminus {\sf{SP}}(FL_1/L_1(T))$ contains infinitely many $K$-regular extensions.

\vspace{0.5mm}

\noindent
{\rm{(2)}} Let $K \supseteq k$ be an algebraically closed overfield and $L_2=K((V))(U)$. Assume $G$ has a non-cyclic abelian subgroup. Then ${\sf{R}}_G(L_2) \setminus {\sf{SP}}(FL_2/L_2(T))$ contains infinitely many $K((V))$-regular extensions.

\vspace{0.5mm}

\noindent
{\rm{(3)}} Let $L_3=k(U)$. Assume either one of the following two conditions holds:

\vspace{0.5mm}

{\rm{(a)}} $G$ is cyclic of even order, $r=2$, and ${\bf{t}} \not \subset \mathbb{P}^1(k)$,

\vspace{0.5mm}

{\rm{(b)}} $G$ is dihedral of order $2n$ with $n \geq 3$ odd, $r =3$, and ${\bf{t}} \not \subset \mathbb{P}^1(k)$.

\vspace{0.5mm}

\noindent
Then ${\sf{R}}_G(L_3) \setminus {\sf{SP}}(FL_3/L_3(T))$ contains infinitely many $k$-regular extensions. 
\end{theorem}

The extensions for which none of the statements of Theorem \ref{list} applies are of genus $0$ (by (1)). Recall that, if $F \cap \overline{k}=k$, the case $g=0$ can occur only in the next situations:

\noindent
$\bullet$ $G$ is cyclic and ${\bf{e}}=(|G|,|G|)$,

\noindent
$\bullet$ $G$ is dihedral and ${\bf{e}}=(2,2,|G|/2)$,

\noindent
$\bullet$ $G=A_4$ and ${\bf{e}}=(2,3,3)$,

\noindent
$\bullet$ $G=S_4$ and ${\bf{e}}=(2,3,4)$,

\noindent
$\bullet$ $G=A_5$ and ${\bf{e}} = (2,3,5)$.

\noindent
Hence, taking now (2) and (3) into account, we obtain that, if $F \cap \overline{k}=k$, the only cases for which none of the statements of Theorem \ref{list} applies are the following ones:

\noindent
{\rm{(a)}} $G$ is cyclic of even order, $r=2$, and ${\bf{t}} \subset \mathbb{P}^1(k)$,

\noindent
{\rm{(b)}} $G$ is cyclic of odd order and $r=2$,

\noindent
{\rm{(c)}} $G$ is dihedral of order $2n$ with $n \geq 3$ odd, $r=3$, and ${\bf{t}} \subset \mathbb{P}^1(k)$.

\begin{proposition} \label{list2}
Assume $F \cap \overline{k}=k$ and that {\rm{(a)}} or {\rm{(b)}} or {\rm{(c)}} holds. Then, for every field $L \supseteq k$ and every $E/L$ in ${\sf{R}}_{\leq G}(L)$, we have $E=(FL)_{t_0}$ for infinitely many $t_0 \in \mathbb{P}^1(L)$.
\end{proposition}

\subsection{Proof of Theorem \ref{list}{\rm{(1)}}} \label{ssec:proof1_2}

As $K$ is ample, ${\sf{R}}_G(L_1)$ contains infinitely many pairwise linearly disjoint $K$-regular extensions (see \cite[Main Theorem A]{Pop96}). If $F \cap \overline{k} \not=k$, at most one of these is in ${\sf{SP}}(FL_1/L_1(T))$. Hence, assume $F \cap \overline{k}=k$. Then, as $g \geq 1$ and $K$ is ample, \cite[Theorem 3.7(a-1)]{DKLN18} yields infinitely many $K$-regular extensions $E/L_1 \in {\sf{R}}_G(L_1)$ each of which satisfies $E \not=(FL_1)_{t_0}$ for any $t_0 \in L_1 \setminus K$. Pick such an $E/L_1$ and assume $E=(FL_1)_{t_0}$ for some $t_0 \in \mathbb{P}^1(L_1)$. Then $t_0 \in \mathbb{P}^1(K)$ and Lemma \ref{spec} gives $E=(FK)_{t_0} L_1$. As $E \cap \overline{K} = K$, we get $(FK)_{t_0}=K$ and so $E=L_1$, a contradiction.

\subsection{Proof of Theorem \ref{list}{\rm{(2)}}} \label{ssec:proof1_3}

Set $M=K((V))$ (so $L_2=M(U) = K((V))(U)$). For each $u \in M$, denote by $\frak{P}_{u}$ the prime ideal of $M[U]$ generated by $U-u$.

We will need the following two lemmas. The first one is a function field analog of \cite[Proposition 6.3]{KLN19} (which is stated over number fields):

\begin{lemma} \label{cyclic}
Assume $F \cap \overline{k}=k$. For each $t_0 \in \mathbb{P}^1(L_2)$ and all but finitely many $u \in M$ (not depending on $t_0$), the Galois group of the completion at $\frak{P}_u$ of $(FL_2)_{t_0}/ L_2$ is cyclic.
\end{lemma}

\begin{proof}
The proof is similar to that in the number field case, and relies on \cite[Theorem 4.1]{KLN19} (the main result of that paper). For the convenience of the reader, we offer a proof, with the needed adjustments. Let $u \in M$ and $t_0 \in \mathbb{P}^1(L_2)$. If $({F}L_2)_{t_0}/L_2$ is unrami\-fied at $\frak{P}_{u}$, then the Galois group of its completion at $\frak{P}_{u}$ is cyclic (as $M = K((V))$ with $K$ algebraically closed of characteristic 0). We may then suppose $({F}L_2)_{t_0}/L_2$ is ramified at $\frak{P}_{u}$. In particular, $t_0 \not \in {\bf{t}}$. Indeed, if $t_0 \in {\bf{t}}$, then $t_0 \in \mathbb{P}^1(\overline{k})$. By Lemma \ref{spec}, we would have $({F}L_2)_{t_0} = ({F}\overline{k})_{t_0} L_2=L_2,$ which cannot happen as $({F}L_2)_{t_0}/L_2$ ramifies at $\frak{P}_{u}$. Up to dropping finitely many values of $u$ (depending only on ${F}L_2/L_2(T)$), we may use the Specialization Inertia Theorem of \cite[\S2.2]{Leg16} to get that $t_0$ meets some branch point of ${F}/k(T)$, say $t$, modulo $\frak{P}_{u}$ (see \cite[Definition 2.2]{Leg16}). As above, $({F}L_2)_{t} = L_2$. Let $I_t$ be the inertia group of ${F}k(t)/k(t)(T)$ at $\langle T-t \rangle$. Up to dropping finitely many values of $u$ (depending only on ${F}L_2/L_2(T)$), \cite[Theorem 4.1]{KLN19} yields that the Galois group of the completion at $\frak{P}_{u}$ of $({F}L_2)_{t_0}/L_2$ embeds into $I_{t}$. As $I_{t}$ is cyclic, we are done.
\end{proof}

\begin{lemma}\label{lem:patching}
For every $u \in M$, there exists an $M$-regular extension $E/L_2\in {\sf{R}}_G(L_2)$ whose completion at $\mathfrak{P}_u$ has a non-cyclic abelian Galois group.
\end{lemma}

\begin{proof}
By a linear change of the variable $U$, we may without loss of generality assume that $\frak P_u = \langle U \rangle$. To prove the lemma, we follow the construction and patching methods of \cite[\S 4]{HHK11}, and adjust these to our setup. Let $X=\mathbb P^1_K$ be the closed fibre of $\hat X=\mathbb P^1_{K[[V]]}$. Our prime $\frak P_u$ corresponds to a maximal ideal $\langle V,U \rangle \lhd K[[V]][U]$ whose image in $K[U]$ is $\langle U \rangle $, the prime corresponding to the point $0 \in \mathbb P^1_K$ of the closed fibre.  

Let $\mathcal C$ be the set of cyclic subgroups of $G$, and $S\subseteq \mathbb A^1_K\subseteq X$ a finite set of points\footnote{Note that, although in \cite{HHK11} the set $S$ is chosen specifically, \cite[Proposition 3.4]{HHK13} shows that the set $S$ can be chosen to be an arbitrary finite set.} consisting of $u_0 = 0 \in \mathbb A^1_K$ and distinct non-zero points $u_c \in \mathbb A^1_K$, $c \in \mathcal C$. For $u_c \in S$, let $F_c=K((V,U-u_c))$ denote the fraction field of the complete local ring $\hat R_c = K[[V,U-u_c]]$ at $u_c$. For the (open) complement $O = X\setminus S$, consider the $V$-adic completion of the subring of  functions on $\hat X$ that are regular on $O$, and let $F_O$ be its fraction field. For $u_c\in S$, also consider the localization of $\hat R_c$ at the prime $\langle V \rangle$ (corresponding to a branch in \cite{HHK11}), and let $F_{\wp(c)}=K((U-u_c))((V))$ be the fraction field of its $V$-adic completion. 

Consider the inverse system $I$ whose objects are the above fields $F_O$ and $F_c$, $F_{\wp(c)}$ for $u_c\in S$, and whose morphisms are the natural inclusions $F_c \subseteq F_{\wp(c)}$ and $F_O \subseteq F_{\wp(c)}$ for $u_c\in S$. A collection of $G$-Galois extensions $E_{\xi}/F_\xi$, $\xi \in I$ of \'etale algebras, together with isomorphisms $E_{O} \otimes_{F_O} F_{\wp(c)} \ra E_{\wp(c)}$, $E_{c} \otimes_{F_c} F_{\wp(c)} \ra E_{\wp(c)}$ for all $u_c\in S$, is called a {\it patching data} (of $G$-Galois \'etale algebras). 

We construct the fields $E_{\xi}$, $\xi\in I$ following the proof of \cite[Proposition 4.4]{HHK11}. Let $E_O=F_O^{|G|}$ and $E_{\wp(c)} = F_{\wp(c)}^{|G|}$, $u_c\in S$ be  split $G$-Galois extensions. It now remains to define the $G$-Galois extensions $E_c/F_c$ whose underlying field $E_c'$ is contained in $F_{\wp(c)}$, and hence $E_c\otimes_{F_c} F_{\wp(c)}\cong F_{\wp(c)}^{|G|} = E_{\wp(c)}$, giving the desired isomorphisms. 

Put $f_c=U-u_c\in \hat R_c$ and set
$$a_c = \frac{f_c}{f_c-V} \text{ for }u_c\in S, \text{ and } b_0=\frac{U-V^2}{U-V-V^2}.$$
For $c \in \mathcal C$, let $E_c'=F_c(\sqrt[m]{a_c})$ for $m = |c|$. As $a_c$ is not a $d$-th power for any $d>1$, Kummer theory yields ${\rm{Gal}}(E'_c/F_c) = c$. Moreover, since $E_{\wp(c)}$ is complete with respect to $\langle V \rangle$, the element $a_c$ is an $m$-th power in $F_{\wp(c)}$ by Hensel's lemma, and hence $E_c'\subseteq F_{\wp(c)}$. Then let $E_{c}/F_{c}$ be the $G$-Galois extension induced from the $c$-Galois field extension $E'_c/F_c$.

Finally, let $H \cong \Zz/q\Zz \times \Zz/q\Zz$ be a non-cyclic abelian subgroup of $G$ for a prime $q$, and $E'_{0} = F_0(\sqrt[q]{a_0},\sqrt[q]{b_0})$. Using Kummer theory as before, $F_0(\sqrt[q]{a_0})/F_0$ and $F_0(\sqrt[q]{b_0})/F_0$) both have Galois group $ \Zz/q\Zz$. The extension $F_0(\sqrt[q]{a_0})/F_0$ is ramified at the prime $\langle U \rangle \lhd \hat R_0$ while $F_0(\sqrt[q]{b_0})/F_0$ is unramified there. Hence, $F_0(\sqrt[q]{a_0})$ and $F_0(\sqrt[q]{b_0})$ are linearly disjoint over $F_0$, and $E'_0/F_0$ is a Galois extension of fields with Galois group $H$. Applying Hensel's lemma as before, one has $E'_0\subseteq F_{\wp(0)}$. Then let $E_0/F_{0}$ be the extension induced from $E'_0/F_0$. As $E'_c\subseteq F_{\wp(c)}$, $u_c\in S$, we conclude that $E_\xi$, $\xi\in I$ is a patching data. 

By \cite[Theorem 4.1]{HHK11}, there exists a $G$-Galois extension  $E/F$ of \'etale algebras such that $E\otimes_{L_2} F_c$ is isomorphic to $E_{c}$ for every $u_c\in S$. Moreover, identifying the two via this isomorphism, $E_c$ is generated by $E$ and $F_c$, and the subalgebra $E'\otimes_{L_2}F_c\subseteq E_c$, generated by $F_c$ and the underlying field $E'$ of $E$, contains a conjugate copy of $E'_c$, for every $u_c\in S$. We claim that $E/L_2$ is an $M$-regular field extension. Put $G'=\Gal(E'/L_2)$. For every $c\in\mathcal C$, the subgroup $c$ of $G$ is the stabilizer of $E'_{c}$ under the action of $G$. As $E'\otimes_{L_2}F_{c}$ contains a conjugate copy of $E'_{c}$ and is stabilized by $G'$, we get that $G'$ contains a conjugate of $c$. As $G'$ contains a conjugate of every cyclic subgroup $c$ of $G$, Jordan's theorem (see \cite{Jor72}) implies that $G'=G$, and hence $E'=E$ is a field. To show that $E$ is $M$-regular, let $\oline G \trianglelefteq G$ be the subgroup fixing the constant field $E\cap \oline M$. As $E_{c}'$ is a compositum of $E'$ and $F_{c}$ by Remark \ref{rem:comp} and $E_c'/F_c$ is $M$-regular, we get that $\oline G$ contains a conjugate of $c$. As $\oline G$ is normal, $c\leq \oline G$ for every $c\in\mathcal C$, and hence $\oline G=G$, as claimed.

Finally, since the completion of $E$ at $\frak P_u$ is the field underlying $E\otimes_{L_2} M((U))$, and since the field $E'_{0}$ is a compositum of $E$ and $F_{0}$, the completion of $E$ at $\langle U \rangle$ is isomorphic to the underlying field $M((U))(\sqrt[q]{a_0},\sqrt[q]{b_0})$ of $E'_0\otimes_{F_0} M((U))$. Now, $M((U))(\sqrt[q]{a_0})/M((U))$ has Galois group $\Zz/q\Zz$ and is totally ramified. On the other hand, $M((U))(\sqrt[q]{b_0})/M((U))$ is unramified and hence linearly disjoint from $M((U))(\sqrt[q]{a_0})/M((U))$. We claim that ${\rm{Gal}}(M((U))(\sqrt[q]{b_0})/M((U))) = \Zz/q\Zz$. Given the claim, the above shows that the extension $M((U))(\sqrt[q]{a_0},\sqrt[q]{b_0})/M((U))$ is of group $H$, which is non-cyclic abelian, as desired.

To prove the claim, observe that $b_0\equiv V/(V+1)$ modulo $U$, and put $\oline{b_0}=V/(V+1)$. As $M(\sqrt[q]{\oline{b_0}})/M$ is totally ramified at $\langle V \rangle$, we deduce that ${\rm{Gal}}(M(\sqrt[q]{\oline{b_0}})/M) = \Zz/q\Zz$. Hensel's lemma then implies that $M((U))(\sqrt[q]{b_0})/M((U))$ also has Galois group $\Zz/q\Zz$. 
\end{proof}

\begin{proof}[Proof of Theorem \ref{list}{\rm{(2)}}]
If $F \cap \overline{k} \not=k$, then ${\rm{Gal}}(FL_2/L_2(T)) \not=G$. Hence, ${\sf{R}}_{G}(L_2) \cap {\sf{SP}}(FL_2/L_2(T)) = \emptyset$. But, as $K$ is algebraically closed of characteristic 0, Riemann's existence theorem yields ${\sf{R}}_G(K(U)) \not=\emptyset$ and so ${\sf{R}}_G(L_2)$ contains infinitely many $M$-regular extensions. Hence, we may assume $F \cap \overline{k} =k$. Lemmas \ref{cyclic} and \ref{lem:patching} then yield that ${\sf{R}}_G(L_2) \setminus {\sf{SP}}(FL_2/L_2(T))$ contains infinitely many $M$-regular extensions, as needed.
\end{proof}

\subsection{Proof of Theorem \ref{list}{\rm{(3)}}} \label{ssec:proof1_4}

First, recall that, if a finite group $G$ is a regular Galois group over an infinite field $k$, i.e., if there is a $k$-regular Galois extension of $k(U)$ of group $G$, then applying suitable M\"obius transformations on $U$ leads to infinitely many $k$-regular Galois extensions of $k(U)$ of group $G$ such that the branch point sets of any two such extensions are disjoint. The Riemann--Hurwitz formula then shows that these $k$-regular Galois extensions of $k(U)$ are pairwise linearly disjoint. In the present situation, $G$ is cyclic or dihedral of order $2n$ with $n \geq3$ odd, and $k$ is of characteristic 0. Since abelian groups and dihedral groups are regular Galois groups over all fields, we get that ${\sf{R}}_G(L_3)$ contains infinitely many pairwise linearly disjoint $k$-regular extensions. If $F/k(T)$ is not $k$-regular, at most one of these can be in ${\sf{SP}}(F(U)/L_3(T))$. Hence, assume $F \cap \overline{k}=k$.

Now, assume $G=\Zz/n\Zz$ for some even $n \geq 2$ and $r=2$. As $n$ is even, there is an $L_3$-regular Galois extension of $L_3(T)$ of group $G$ with a branch point in $\mathbb{P}^1(L_3)$, and with another branch point of ramification index $n$. Then, by \cite[Corollary 3.4]{Leg16}, there is a prime $\frak{Q}$ of $k[U]$ such that, for all but finitely many $u$ in $k$, there is $E_u/L_3 \in {\sf{R}}_G(L_3)$ which ramifies at $\frak{P}_u=\langle U-u \rangle$, and whose ramification index at $\frak{Q}$ is $n$. In particular, $E_u/L_3$ is $k$-regular (by the last condition). Suppose $E_u/L_3 \in {\sf{SP}}(F(U)/L_3(T))$ for infinitely many $u \in k$. Without loss, we may assume $\infty \not \in {\bf{t}}$. For $i \in \{1,2\}$, denote the minimal polynomial of $t_i$ over $k$ by $m_i(T)$. Then, by \cite[Corollary 2.12 and Remark 3.11]{Leg16}, the reduction modulo $\frak{P}_u$ of $m_1(T) m_2(T) \in k[U][T]$ has a root in the residue field $k[U] /\frak{P}_u$ for some $u \in k$. As this residue field is $k$, $m_1(T)m_2(T)$ has a root in $k$. Hence, by the Branch Cycle Lemma (see \cite{Fri77} and \cite[Lemma 2.8]{Vol96}), $t_1$ and $t_2$ are in $\mathbb{P}^1(k)$.

Finally, assume $G$ is dihedral of order $2n$ for some odd $n \geq 3$ and $r=3$. As $n$ is odd, the ramification indices $e_1$, $e_2$, and $e_3$ are 2, 2, and $n$, respectively (up to reordering). In particular, by the Branch Cycle Lemma (and as $n \not=2$), $t_3$ is in $\mathbb{P}^1(k)$. By \cite[\S16.2 and Proposition 16.4.4]{FJ08}, every $k$-regular extension in ${\sf{R}}_{\Zz/2\Zz}(L_3)$ embeds into a $k$-regular extension in ${\sf{R}}_{G}(L_3)$. Hence, if all but finitely many $k$-regular extensions in ${\sf{R}}_{G}(L_3)$ are in ${\sf{SP}}(F(U)/L_3(T))$, then, as $G$ has a unique subgroup of index 2, all but finitely many $k$-regular extensions in ${\sf{R}}_{\Zz/2\Zz}(L_3)$ are specializations of the quadratic subextension of $F(U)/L_3(T)$. As the latter has only two branch points (namely, $t_1$ and $t_2$), a similar argument as in the cyclic case yields that these branch points have to be in $\mathbb{P}^1(k)$.

\subsection{Proof of Proposition \ref{list2}} \label{ssec:proof2}

Assume $F \cap \overline{k}=k$ and one of the following holds:

\noindent
{\rm{(a)}} $G$ is cyclic of even order, $r=2$, and ${\bf{t}} \subset \mathbb{P}^1(k)$,

\noindent
{\rm{(b)}} $G$ is cyclic of odd order and $r=2$,

\noindent
{\rm{(c)}} $G$ is dihedral of order $2n$ with $n \geq 3$ odd, $r=3$, and ${\bf{t}} \subset \mathbb{P}^1(k)$.

\noindent
Let $L \supseteq k$ and $E/L \in {\sf{R}}_{\leq G}(L)$. By the twisting lemma (see \cite{Deb99a}), there is a $k$-regular extension $(FL)_E/L(T)$ such that $(FL)_E \overline{L}=F\overline{L}$ and such that, given $t_0 \in \mathbb{P}^1(L) \setminus {\bf{t}}$, if there is a prime ideal lying over $\langle T-t_0 \rangle$ in $(FL)_E/L(T)$ with residue degree 1, then $E/L = (FL)_{t_0}/L$. In each case, the genus of $F$ is $0$ (if (c) holds, this follows from ${\bf{e}}$ being $(2,2,|G|/2)$). Hence, $(FL)_E$ has genus 0 as well. It then suffices to find $t\in \mathbb{P}^1(L)$ for which there is a prime ideal lying over $\langle T-t \rangle $ in $(FL)_E/L(T)$ with residue degree 1.

If (a) holds, then the unique prime ideal lying over $\langle T-t_1 \rangle$ in $(FL)_E/L(T)$ has residue degree 1. If (b) holds, the desired conclusion follows from $G$ being of odd order and the genus being 0; see, e.g., end of Page 1 of \cite{Ser92}. Finally, assume (c) holds. As already seen, the ramification indices $e_1$, $e_2$, and $e_3$ are 2, 2, and $n$, where $|G|=2n$, respectively (up to reordering). As $\{t_1, t_2\} \subset \mathbb{P}^1(k)$, we may assume that the quadratic subfield of $F$ is $k(\sqrt{T})$ (up to applying a suitable change of variable). Hence, there is $d \in L \setminus \{0\}$ such that $(FL)_E$ contains $L(\sqrt{dT})$. Set $Y=\sqrt{dT}$. The extension $(FL)_E/L(Y)$ is of degree $n$ and it has only two branch points; it is then Galois of group $\Zz/n\Zz$ and of genus 0. As $n$ is odd, there is $y_0 \in L$ such that the specialization of $(FL)_E/L(Y)$ at $y_0$ is $L/L$. Hence, there is a prime ideal lying over $\langle T-(y_0)^2/d \rangle $ in $(FL)_E/L(T)$ with residue degree 1.

\subsection{Proofs of Theorem \ref{thm:intro_2} and Corollary \ref{cor:list}} \label{ssec:proof_0}

We conclude this section by explai\-ning how Theorem \ref{thm:intro_2} and Corollary \ref{cor:list} follow from Theorem \ref{list} and Proposition \ref{list2}.

We start with the following consequence, of which Conclusion (1) is Theorem \ref{thm:intro_2} and Conclusion (2) is mentioned in the abstract:

\begin{corollary} \label{thm:intro_2bis}
Let $k$ be of characteristic 0, let $U, V$ be indeterminates, and let $P(T,Y) \in k[T][Y]$ be a monic separable polynomial of group $G$ and splitting field $F$ over $k(T)$.

\vspace{0.5mm}

\noindent
{\rm{(1)}} Assume $P(T,Y)$ is not generic. Then either $P(T,Y)$ is not $k(U)$-parametric or $P(T,Y)$ is not ${K}((V))(U)$-parametric for any algebraically closed overfield $K \supseteq k$.

\vspace{0.5mm}

\noindent
{\rm{(2)}} Assume $P(T,Y)$ is not generic and $k$ is algebraically closed. Then $P(T,Y)$ is not ${K}((V))(U)$-parametric for any algebraically closed overfield $K \supseteq k$.

\vspace{0.5mm}

\noindent
{\rm{(3)}} Assume $G$ is neither cyclic nor dihedral of order $2n$ with $n \geq 3$ odd. Then $P(T,Y)$ is not $K((V))(U)$-parametric for any algebraically closed overfield $K \supseteq k$.

\vspace{0.5mm}

\noindent
{\rm{(4)}} If $G \not \subset {\rm{PGL}}_2(\Cc)$, then $P(T,Y)$ is $K(U)$-parametric for no ample overfield $K \supseteq k$.
\end{corollary}

\begin{proof}
(1) Assume $P(T,Y)$ is $k(U)$-parametric and $K((V))(U)$-parametric for some algebraically closed overfield $K \supseteq k$. Then, by Lemma \ref{prop:compare_1}(1), the same holds for $F/k(T)$. As in the proof of Theorem \ref{list}(2) (see the end of \S\ref{ssec:proof1_3}), $F/k(T)$ being $K((V))(U)$-parametric implies that $F/k(T)$ is $k$-regular. Moreover, by Theorem \ref{list}, one of the three conditions stated before Proposition \ref{list2} holds. Hence, by that proposition, we have that, for every overfield $L \supseteq k$ and every $E/L \in {\sf{R}}_{G}(L)$, there exist infinitely many $t_0 \in \mathbb{P}^1(L)$ such that $E=(FL)_{t_0}$. It then remains to use Lemma \ref{prop:compare_1}(2) to conclude that $P(T,Y)$ is generic.

\vspace{0.75mm}

\noindent
(2) The proof is similar to that of (1), except that we have to use that neither Condition (a) nor Condition (b) of Theorem \ref{list}(3) can happen (since $k$ is algebraically closed).

\vspace{0.75mm}

\noindent
(3) If $P(T,Y)$ is $K((V))(U)$-parametric for some algebraically closed overfield $K \supseteq k$, then, as in (1), $F/k(T)$ is $K((V))(U)$-parametric and $k$-regular. Moreover, Theorem \ref{list}(1) and Theorem \ref{list}(2) yield that we are in the situation of Theorem \ref{list}(3) or in that of Proposition \ref{list2}. In both cases, $G$ is cyclic or dihedral of order $2n$ with $n \geq 3$ odd.

\vspace{0.75mm}

\noindent
(4) Assume $P(T,Y)$ is $K(U)$-parametric for some ample overfield $K \supseteq k$. As in the proof of (1), $F/k(T)$ is $K(U)$-parametric. As already recalled in \S\ref{ssec:proof1_2}, ${\sf{R}}_G(K(U))$ contains infinitely many pairwise linearly disjoint extensions. Hence, ${F}/k(T)$ is $k$-regular. Finally, Theorem \ref{list}(1) gives that $F/k(T)$ is of genus 0, and so $G \subset {\rm{PGL}}_2(\Cc)$.
\end{proof}

We finally get to the classification of all the one parameter generic polynomial/extensions over a given field of characteristic zero:

\begin{corollary} \label{list ext}
Let $k$ be of characteristic 0 and $P(T,Y) \in k[T][Y]$ a monic separable polynomial of group $G$ and splitting field $F$ over $k(T)$. Denote the branch point number (resp., branch point set) of $F/k(T)$ by $r$ (resp., by ${\bf{t}}$). The following three conditions are equivalent:

\vspace{0.5mm}

\noindent
{\rm{(1)}} $F/k(T)$ is generic,

\vspace{0.5mm}

\noindent
{\rm{(2)}} $P(T,Y)$ is generic,

\vspace{0.5mm}

\noindent
{\rm{(3)}} $F \cap \overline{k}=k$ and one of the following three conditions holds:

\vspace{0.25mm}

{\rm{(a)}} $G$ is cyclic of even order $n$ such that $e^{2i\pi/n} \in k$, $r=2$, and ${\bf{t}} \subset \mathbb{P}^1(k)$,

\vspace{0.25mm}

{\rm{(b)}} $G$ is cyclic of odd order $n$ such that $e^{2i \pi/n} + e^{-2i \pi/n} \in k$ and $r=2$,

\vspace{0.25mm}

{\rm{(c)}} {\hbox{$G$ is dihedral of order $2n$ with $n \geq 3$ odd and $e^{2i \pi/n} + e^{-2i \pi/n} \in k$, $r=3$, and ${\bf{t}} \subset \mathbb{P}^1(k)$.}}
\end{corollary}

We need the next lemma, which is classical in inverse Galois theory. The ``only if" part is an immediate consequence of the Branch Cycle Lemma (see \cite{Fri77} and \cite[Lemma 2.8]{Vol96}) and the ``if" part is due to the rigidity method (see, e.g., \cite[Chapter 3]{Vol96}).

\begin{lemma} \label{lemma:rigid}
Let $k$ be a field of characteristic zero.

\vspace{0.5mm}

\noindent
{\rm{(1)}} Given $n \geq 2$, there is a $k$-regular Galois extension of $k(T)$ of group $\Zz/n\Zz$ and with two branch points if and only if $e^{2i \pi/n} + e^{-2i \pi/n} \in k$; both branch points can be chosen in $\mathbb{P}^1(k)$ if and only if $e^{2i \pi/n} \in k$.

\vspace{0.5mm}

\noindent
{\rm{(2)}} Given $n \geq 3$ odd, there is a $k$-regular Galois extension of $k(T)$ with dihedral Galois group of order $2n$, with three branch points, and all branch points in $\mathbb{P}^1(k)$ if and only if $e^{2i \pi/n} + e^{-2i \pi/n} \in k$. 
\end{lemma}

\begin{proof}[Proof of Corollary \ref{list ext}]
First, (2) $\Rightarrow$ (1) follows from Lemma \ref{prop:compare_1}(1). Now, assume (1) holds. Then, as already seen, $F/k(T)$ is $k$-regular. Moreover, by Theorem \ref{list}, one of the conditions stated before Proposition \ref{list2} holds. Then, by Lemma \ref{lemma:rigid}, (3) holds. Finally, if (3) holds, then $P(T,Y)$ is generic, by Lemma \ref{prop:compare_1}(2) and Proposition \ref{list2}.
\end{proof}

\begin{remark} \label{rk:3.6}
(1) Corollary \ref{list ext} and Lemma \ref{lemma:rigid} give another proof of  Theorem \ref{prop:classification}.

\vspace{0.5mm}

\noindent
(2) In the spirit of Definition \ref{def:main}, say that $F/k(T)$ is {\it{strongly generic}} if it is strongly $L$-parametric for every overfield $L \supseteq k$. Clearly, we have $F/k(T)$ strongly generic $\Rightarrow$ $F/k(T)$ generic. The converse holds by combining Proposition \ref{list2} and Corollary \ref{list ext}.

\vspace{0.5mm}

\noindent
(3) Definition \ref{def:intro} is the definition of generic polynomials of \cite{JLY02}. Variants could have been used. For example, a strong one, used by Kemper (see \cite{Kem01}), requires extensions in ${\sf{R}}_{\leq G}(L)$ to be parametrized. In \cite{DeM83}, DeMeyer even requires every extension in ${\sf{R}}_{\leq G}(L)$ to be realized by a {\it{separable}} specialized polynomial. Lemma \ref{prop:compare_1}(2), Proposition \ref{list2}, and Corollary \ref{list ext} show that, for one parameter polynomials over fields of characte\-ristic 0, the three definitions are equivalent. In particular, we retrieve \cite[Theorem 1]{Kem01} in this case (Kemper's result asserts, more generally, that the first two definitions are equivalent over infinite fields, for polynomials with an arbitrary number of parameters).
\end{remark}

\begin{proof}[Proof of Corollary \ref{cor:list}]
Let $G$ be finite non-trivial and $P(T,Y) \in \Qq[T][Y]$ be monic se\-pa\-rable of group $G$ and splitting field $F$ over $\Qq(T)$. By (2) $\Leftrightarrow$ (3) in Corollary \ref{list ext} (with $k=\Qq$), $P(T,Y)$ is generic if and only if $F \cap \overline{\Qq} = \Qq$ and one of these conditions holds:

\noindent
- $G=\Zz/2\Zz$ and $F/\Qq(T)$ has two branch points, which are $\Qq$-rational,

\noindent
- $G=\Zz/3\Zz$ and $F/\Qq(T)$ has two branch points,

\noindent
- $G=S_3$ and $F/\Qq(T)$ has three branch points, which are $\Qq$-rational.

In the first case, observe next that any $\Qq$-regular quadratic extension of $\Qq(T)$ with two branch points, which are $\Qq$-rational, equals $\Qq(\sqrt{d(T-a)(T-b)})/\Qq(T)$ or $\Qq(\sqrt{d(T-a)})/\Qq(T)$ for some $d \in \Zz$ and $a, b \in \Qq$. All of these are derived from $\Qq(\sqrt{T})/\Qq(T)$, by applying a suitable M\"obius transformation on $T$. Hence, if $G=\Zz/2\Zz$, the polynomial $P(T,Y)$ is generic if and only if $F=\Qq(\sqrt{T})$, up to some M\"obius transformation on $T$.

In the second case, let $F_1/\Qq(T)$ and $F_2/\Qq(T)$ be $\Qq$-regular Galois extensions of group $\Zz/3\Zz$ with two branch points. Fix $j \in \{1, 2\}$. By the Branch Cycle Lemma, the branch points of $F_j/\Qq(T)$ are $\Qq$-conjugate and generate $\Qq(e^{2 i \pi/3})$ over $\Qq$. Moreover, $(F_j)_{t_j}= \Qq$ for some $t_j \in \mathbb{P}^1(\Qq)$. Up to applying $T \mapsto 1/(T-t_j)$, we may assume $t_j=\infty$. Then, up to applying $T \mapsto (T-a)/b$ for some $a,b \in \Qq$ with $b \not=0$, which fixes $\infty$, we may also assume the branch point set of $F_j/\Qq(T)$ is $\{e^{2 i \pi/3}, e^{4 i \pi/3}\}$. Then $F_1 \overline{\Qq} = F_2\overline{\Qq}$. Indeed, we would have otherwise that $F_1F_2 \overline{\Qq}/ \overline{\Qq}(T)$ has Galois group $\Zz/3\Zz \times \Zz/3\Zz$, and so has at least three branch points, which cannot happen. Hence, $[F_1F_2 \overline{\Qq} : \overline{\Qq}(T)]=3$. If $F_1 \not=F_2$, then $[F_1F_2 : \Qq(T)]=9$ and $F_1F_2/\Qq(T)$ has a degree 3 constant subextension. But the latter cannot happen as $(F_1F_2)_{\infty} = (F_1)_{\infty} (F_2)_{\infty} = \Qq$ (see \cite[Lemma 2.4.8]{FJ08} for the first equality). We then have $F_1=F_2$. Hence, a $\Qq$-regular Galois extension of $\Qq(T)$ of group $\Zz/3\Zz$ with two branch points is unique, up to M\"obius transformations on $T$. Let $F'$ be the splitting field of $Y^3 - TY^2 + (T-3)Y + 1$ over $\Qq(T)$. As $F'/\Qq(T)$ is $\Qq$-regular of group $\Zz/3\Zz$, and has two branch points, we are done in the case $G=\Zz/3\Zz$.

Finally, consider the case $G=S_3$. The inertia canonical invariant of a $\Qq$-regular Ga\-lois extension of $\Qq(T)$ of group $S_3$ with 3 branch points is $(C_2, C_2, C_3)$, with $C_n$ the conjugacy class of the $n$-cycles. As $(C_2, C_2, C_3)$ is a rigid triple of rational con\-ju\-ga\-cy classes of the centerless group $S_3$, there is only one $\Qq$-regular Galois extension of $\Qq(T)$ of group $S_3$ with three $\Qq$-rational branch points, up to M\"obius transformation on $T$ (see \cite[Chapters 7 and 8]{Ser92}). Then we are done as, if $F'$ is the splitting field of $Y^3+TY+T$ over $\Qq(T)$, then $F'/\Qq(T)$ is $\Qq$-regular, of group $S_3$, and of branch point set $\{0, \infty, -27/4\}$.
\end{proof}

\section{On Schinzel's problem and its variants} \label{sec:wh}

We investigate the connections between $k$-parametricity and $k(U)$-parametricity, in relation with Schinzel's problem (Question \ref{pb:schinzel}). 

In \S\ref{sec:intro}, we mentioned a close variant of Question \ref{pb:schinzel}. It corresponds to the following diophantine working hypothesis, which is introduced in \cite[\S2.4.2]{Deb18}:

\vspace{2mm}

\noindent
{\rm{(WH)}} {\it{Let $k$ be a number field, and let $f_i:X_i \rightarrow \Pp^1_{k(U)}$, $i=1,\ldots,N$, be $k(U)$-regular covers. Assume that no curve $X_i$ has a $\Cc(U)$-rational point that is unramified \hbox{w.r.t.} the cover $f_i$, $i=1,\ldots,N$. Then, for infinitely many $u_0\in k$, the covers $f_1,\ldots,f_N$ have good reduction at $U=u_0$ and no reduced curve $X_i|_{u_0}$ has a $k$-rational point that is unramified \hbox{w.r.t.} the cover $f_i|_{u_0}:X_i|_{u_0} \rightarrow \Pp^1_{k}$, $i=1,\ldots,N$.}}

\vspace{2mm}

Proposition \ref{prop:role_WH} below summarizes some of the connections between our notions of parametricity and genericity. As already said, if $k$ is a number field, a close variant of the implication ``$k$-parametric $\Rightarrow$ $k(U)$-parametric" holds under (WH). It is given by the implication ``strongly $k$-parametric $\Rightarrow$ weakly $k(U)$-parametric" below.

\begin{definition} \label{def:weakly}
Let $G$ be a finite group, $k$ a subfield of $\Cc$, and $F/k(T)$ a $k$-regular extension in ${\sf{R}}_G(k(T))$. We say that $F/k(T)$ is {\it weakly $k(U)$-parametric} if every $k$-regular extension $E/k(U) \in {\sf{R}}_G(k(U))$ is a specialization of $F(U)/k(U)(T)$, {after base change} $\Cc/k$.
\end{definition}

\begin{proposition} \label{prop:role_WH}
For a field $k$ of characteristic zero and a finite $k$-regular Galois extension of $k(T)$, we have
$$\small
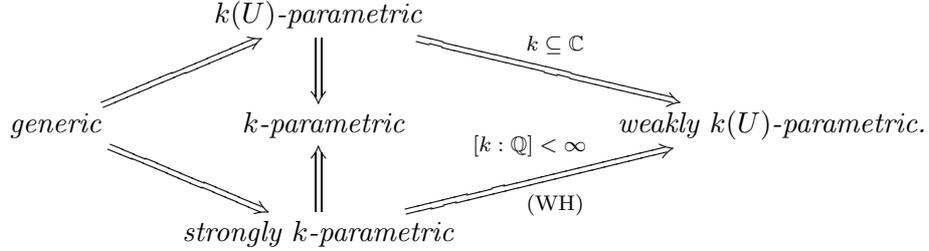
{\xymatrix  {
& \text{$k(U)$-parametric} \ar@{=>}[d] \ar@{=>}[rdr]^[@]{\text{$k \subseteq \Cc$}} && \\
\text{generic} \ar@{=>}[ru] \ar@{=>}[rd] & \text{ $k$-parametric} & & \text{weakly $k(U)$-parametric.}\\
& \text{strongly $k$-parametric} \ar@{=>}[u] \ar@{=>}[rur]^[@]{\text{$[k : \Qq] < \infty$}}_[@]{\text{{\rm{(WH)}}}} & &
}}$$
\end{proposition}

\begin{proof}
The implications ``generic $\Rightarrow$ $k(U)$-parametric" and ``strongly $k$-parametric $\Rightarrow$ $k$-parametric" are clear, while ``generic $\Rightarrow$ strongly $k$-parametric" follows from Remark \ref{rk:3.6}(2). Next, \cite[Remark 2.3]{Deb18} proves ``strongly $k(U)$-parametric $\Rightarrow$ strongly $k$-parametric". With the same arguments as there, we get ``$k(U)$-parametric $\Rightarrow$ $k$-parametric". Now, if $k \subseteq \Cc$, the implication ``$k(U)$-parametric $\Rightarrow$ weakly $k(U)$-parametric" follows from Lemma \ref{spec}. Finally, if $k$ is a number field and (WH) holds, then the implication ``strongly $k$-parametric $\Rightarrow$ weakly $k(U)$-parametric" is \cite[Proposition 2.17(b)]{Deb18}.
\end{proof}

The implication ``strongly $k$-parametric $\Rightarrow$ weakly $k(U)$-parametric", which holds under (WH) and if $k$ is a number field, was used in \cite{Deb18} to produce (conditionally) examples of finite groups with no strongly $k$-parametric extension $F/k(T)$, by first providing exam\-ples of finite groups with no weakly $k(U)$-parametric extension $F/k(T)$. The following theorem suggests, however, that this approach fails in general. 

To word it, we denote by $W(C/k) \in \{\pm 1\}$ the root number of an elliptic curve $C$ over a number field $k$. We refer to, e.g., \cite{Sil09, PRS11} for the definition and, more generally, for more on the terminology of elliptic curves that is used below.

\begin{theorem} \label{thm versus}
Let $k$ be a number field and $Q(T)\in k[T]$ an irreducible degree 3 polynomial such that the elliptic curve $C: Y^2=Q(T)$ fulfills $W(C/k) = -1$, but $W(C/L)=+1$ for every quadratic extension $L/k$. Set $F/\Qq(T) = \Qq(T)(\sqrt{Q(T)})/\Qq(T)$ and $P(U,T,Y)= Y^2 - UQ(T)$. Then, under the Birch and Swinnerton-Dyer conjecture, we have:

\vspace{0.5mm}

\noindent
{\rm{(1)}} $F/\Qq(T)$ is strongly $k$-parametric but neither weakly $k(U)$-parametric nor $k(U)$-parametric,

\vspace{0.5mm}

\noindent
{\rm{(2)}} $Y^2-Q(T)$ is strongly $k$-parametric but not $k(U)$-parametric,

\vspace{0.5mm}

\noindent
{\rm{(3)}} the answer to Question \ref{pb:schinzel} is negative for the field $k$ and the polynomial $P(U,T,Y)$,

\vspace{0.5mm}

\noindent
{\rm{(4)}} the above hypothesis {\rm{(WH)}} fails for the number field $k$ and the sole $k(U)$-regular Galois cover $X \rightarrow \mathbb{P}^1_{k(U)}$ given by the polynomial $P(U,T,Y)$.
\end{theorem}

\begin{proof}[Proof of Theorem \ref{thm versus}]
Under the Birch and Swinnerton-Dyer conjecture, and by our assumption on $W(C/k)$, the elliptic curve $C$ has odd rank over $k$, and so infinitely many $k$-rational points. Similarly, for every non-square $u_0 \in k$, the twisted elliptic curve $C_{u_0}:Y^2=u_0Q(T)$ has positive rank, and so infinitely many $k$-rational points. See \cite{DD09} for more details. Hence, the next two statements (which are equi\-valent) hold:

\vskip 0,5mm

\noindent
(a) {\it{for each $u_0 \in k^*$, the polynomial $P(u_0, T, Y)$ has a zero $(t,y)$ in $k^2$ such that $y \not=0$,}}

\vspace{0.5mm}

\noindent
(b) {\it{every trivial or quadratic extension of $k$ is the splitting field over $k$ of some separable polynomial $Y^2-Q(t_0)$ with $t_0 \in k$}}.

\vskip 0,5mm

Now, we have:

\vskip 0,5mm

\noindent
(c) {\it{given a field $L$ of characteristic zero, $L(\sqrt{U})/L(U) \not \in {\sf{SP}}(FL(U)/L(U)(T))$.}}

\vspace{0.5mm}

\noindent
Indeed, let $L$ be a field of characteristic 0. Since $F/\Qq(T)$ has 4 branch points while $\Qq(\sqrt{T})/\Qq(T)$ has only 2, \cite[Theorem 2.1]{Deb18} yields $L(\sqrt{U}) \not= (FL(U))_{t_0}$ for every $t_0 \in L(U) \setminus L$. Then use Lemma \ref{spec} as in \S\ref{ssec:proof1_2} to rule out the constant specializations at points $t_0\in \Pp^1(L)$.

Next, (c) is equivalent to the following:

\vskip 0,5mm

\noindent
(d) {\it{for $L$ of characteristic zero, $P(U,T,Y)$ has no zero $(t,y) \in L(U)^2$ such that $y \not=0$.}}

\vskip 0.5mm

\noindent
In particular, we have:

\vskip 0,5mm

\noindent
(e) {\it{the polynomial $P(U,T,Y)$ has no zero in $k(U)^2$.}}

\vskip 0,5mm

\noindent
Indeed, suppose $P(U,T,Y)$ has such a zero $(t,y)$. As $Q(T)$ is assumed irreducible over $k$, we have that $t$ is not a root of $Q(T)$. Hence, $y\not= 0$, which cannot happen by (d).

Finally, by (b) and (c), $F/\Qq(T)$ is strongly $k$-parametric but not weakly $k(U)$-parametric. The (weaker) conclusion that it is not $k(U)$-parametric then follows from Proposition \ref{prop:role_WH}. Now, (2) follows from (b), (1), and Lemma \ref{prop:compare_1}(1). Next, (3) follows from (a) and (e).  As to (4), it basically follows from (1) and Proposition \ref{prop:role_WH} (in fact, from (a) and (d)).
\end{proof}

We now explain how Theorem \ref{thm:main4} follows from Theorem \ref{thm versus}:

\begin{proof}[Proof of Theorem \ref{thm:main4}]
Let $Q(T) \in \Qq[T]$ be a degree 3 separable polynomial such that the elliptic curve $C/\Qq$ given by $Y^2=Q(T)$ has complex multiplication by $k_0=\Qq(\sqrt{-m})$ for some $m \in \{11,19, 43, 67, 163\}$. As $k_0 \not= \Qq(\sqrt{-1}), \Qq(\sqrt{-2}), \Qq(\sqrt{-3})$, we may apply \cite[Corollary 4.10]{LY20} to get that there exist infinitely many quadratic number fields $k$ such that $W(C/k)=-1$ and $W(C/L) = +1$ for every quadratic extension $L$ of $k$. Moreover, since $C/\Qq$ has complex multiplication by $k_0=\mathbb{Q}(\sqrt{-m})$ for some $m\in \{11,19,43,67,163\}$, the elliptic curve $C/\Qq$ has trivial $\mathbb{Q}$-torsion (see the table in \cite{Ols74}). In particular, the triviality of the rational $2$-torsion subgroup is equivalent to the irreducibility of $Q(T)$ over $\Qq$, and so over every quadratic number field. Consequently, there exist infinitely many quadratic number fields $k$ such that 
the elliptic curve $C/k$ fulfills the assumptions of Theorem \ref{thm versus}, thus yielding the assertion.
\end{proof}

\begin{remark}\label{rem:versus}
(1) First explicit examples of elliptic curves $C$ and number fields $k$ as in Theorem \ref{thm versus} were given in \cite{DD09}, where they are called ``lawful evil" elliptic curves. See \cite[Theorem 4.9]{LY20} for an even more general construction of such curves $C$ and fields $k$. The explicit example given right after the statement of Theorem \ref{thm:main4} is taken from \cite[Example 4.12(ii)]{LY20}.
	
\vspace{0.75mm}
	
\noindent
(2) In the context of Theorem \ref{thm versus}, the polynomial $Q(T)$ is separable of degree 3. Hence, the $\Cc(U)$-curve $P(U,T,Y)=0$ is of genus $1$. It remains plausible that (WH) holds if $f_1, \dots, f_N$ are all of genus $\geq 2$, which would yield that any given finite $k$-regular Galois extension $F/k(T)$ of genus $\geq 2$ which is not weakly $k(U)$-parametric is actually not strongly $k$-parametric. Similarly, it is plausible that the answer to Question \ref{pb:schinzel} is affirmative for $\Cc(U)$-curves $P(U,T,Y)=0$ of genus at least 2.
\end{remark}

As recalled in Proposition \ref{prop:role_WH}, if a $k$-regular Galois extension of $k(T)$ is $k(U)$-parametric, then it is $k$-parametric. Theorem \ref{thm versus}(1) shows that the converse fails (conditionally) over number fields. Here is another counter-example, unconditional, over Laurent series fields:

\begin{proposition} \label{prop:versus}
Let $k$ be algebraically closed of characteristic zero and $G$ a finite group. 

\vspace{0.5mm}

\noindent
{\rm{(1)}} There exists $F/k(T) \in {\sf{R}}_G(k(T))$ fulfilling the following. Let $K \supseteq k$ be algebraically closed and $L=K((V))$. Then $F/k(T)$ is strongly $L$-parametric. More precisely, given $E/L \in {\sf{R}}_{\leq G}(L)$, we have $E=(FL)_{t_0}$ for infinitely many $t_0 \in \mathbb{P}^1(L)$.

\vspace{0.5mm}

\noindent
{\rm{(2)}} If $G$ is neither cyclic nor dihedral of order $2n$ with $n \geq 3$ odd, then, for $L=K((V))$ where $K$ is any algebraically closed field  containing $k$, the extension $F/k(T)$ is not $L(U)$-parametric.
\end{proposition}

\begin{lemma} \label{sit}
Let $k$ be of characteristic zero, $K \supseteq k$ an algebraically closed overfield, $G$ a finite group, and $L=K((V))$. Then a given $k$-regular extension $F/k(T) \in {\sf{R}}_G(k(T))$ with inertia canonical invariant $(C_1,\dots,C_r)$ is strongly $L$-parametric if and only if 

\vspace{0.5mm}

\noindent
{\rm{($*$)}} for each element order $n$ in $G$, there is $i \in \{1,\dots,r\}$ such that elements of $C_i$ have order divisible by $n$.

\vspace{0.5mm}

\noindent
Moreover, if {\rm{($*$)}} holds, then, given $E/L \in {\sf{R}}_{\leq G}(L)$, there exist infinitely many $t_0 \in \mathbb{P}^1(L)$ such that $E=(FL)_{t_0}$.
\end{lemma}

\begin{proof}
The set ${\sf{R}}_{\leq G}(L)$ precisely consists of all the extensions of the form $L(\sqrt[n]{V})/L$, where $n$ is any element order in $G$. As such an extension $L(\sqrt[n]{V})/L$ is totally ramified of index $n$ at the unique maximal ideal $\frak{P}$ of $K[[V]]$, a given $k$-regular extension ${F}/k(T) \in {\sf{R}}_G(k(T))$ is strongly $L$-parametric if and only if ${F}L/L(T)$ has a specialization at some $t_0 \in \mathbb{P}^1(L)$, which is not a branch point of $F/k(T)$, of ramification index $n$ at $\frak{P}$, for each element order $n$ in $G$.

Firstly, assume ($*$) holds. Let $n$ be an element order in $G$. Pick $i \in \{1,\dots,r\}$ such that the order $e$ of every element of $C_i$ is a multiple of $n$, and set $e=nm$. By \cite[Theorem 3.1]{Leg16}, there are infinitely many $t_0 \in L$ such that the inertia group at $\frak{P}$ of $({F}L)_{t_0}/L$ is generated by an element of $C_i^m$. In particular, the ramification index at $\frak{P}$ of such a specialization is $n$. Hence, ${F}/{k}(T)$ is strongly $L$-parametric. Conversely, assume ${F}/k(T)$ is strongly $L$-parametric. Let $n$ be an element order in $G$. By the above characterization, ${F}L/L(T)$ has a specialization of ramification index $n$ at $\frak{P}$. Then, by the Specialization Inertia Theorem, the inertia canonical invariant of $F/k(T)$ contains the conjugacy class of an element of $G$ of order divisible by $n$. Hence, ($*$) holds.
\end{proof}

\begin{proof}[Proof of Proposition \ref{prop:versus}]
By Riemann's existence theorem, there is $F/k(T) \in {\sf{R}}_G(k(T))$ whose inertia canonical invariant contains the conjugacy class of every element of $G \setminus \{1\}$. In particular, $F/k(T)$ fulfills Condition ($*$) of Lemma \ref{sit}. Hence, $F/k(T)$ is strongly $L$-parametric, for $K \supseteq k$ algebraically closed and $L=K((V))$, and, given $E/L \in {\sf{R}}_{\leq G}(L)$, there are infinitely many $t_0 \in \mathbb{P}^1(L)$ with $E=(FL)_{t_0}$. Finally, if $G$ is neither cyclic nor dihedral of order $2n$ with $n \geq 3$ odd, $F/k(T)$ is not $L(U)$-parametric by Theorem \ref{list} and the subsequent paragraph.
\end{proof}

\begin{remark} \label{rk:PAC}
Using Lemma \ref{prop:compare_1} yields this polynomial analog of Proposition \ref{prop:versus}:

\vspace{1mm}

\noindent
{\it{Let $k$ be algebraically closed of characteristic 0 and $G$ a finite group. There is a monic separable polynomial $P(T,Y) \in k[T][Y]$ of group $G$ such that, for $K \supseteq k$ algebraically closed and $L=K((V))$, the polynomial $P(T,Y)$ is strongly $L$-parametric. Furthermore, if $G$ is neither cyclic nor dihedral of order $2n$ with $n \geq 3$ odd, then $P(T,Y)$ is not $L(U)$-parametric.}}
\end{remark}

\section{Polynomials with more variables}

We conclude with several remarks on polynomials with more than one variable. The first one compares a single parametric polynomial with a finite ``parametric set". 

\begin{remark}\label{rem:finite}
As already used in the proof of Lemma \ref{prop:compare_2}, it is well-known that, over infinite fields, using more than one polynomial is redundant in the setup of generic polynomials. Namely, suppose $P_1(T_1, \dots, T_n,Y), \ldots,$ $P_r (T_1, \dots, T_n,Y) \in k[T_1,\dots,T_n][Y]$ are finitely many monic separable polynomials of group $G$ over $k(T_1, \dots, T_n)$ fulfilling this: for every overfield $L \supseteq k$ and every $E/L \in {\sf{R}}_G(L)$, there are $i \in \{1, \dots, r\}$ and $(t_1,\ldots,t_n)\in L^n$ such that $E$ is the splitting field over $L$ of $P_i(t_1,\ldots,t_n,Y)$. By \cite[Corollary 1.1.6]{JLY02}, it follows that at least one of the polynomials $P_i(T_1, \dots, T_n, Y)$ has to be generic itself.

On the other hand, the analogous property for parametric polynomials fails in general, e.g., for $G=\mathbb Z/8\mathbb Z$ and $k=\mathbb Q(\sqrt{17})$. Namely, there exist $n \geq 1$ and finitely many monic separable polynomials $P_1(T_1, \dots, T_n, Y), \ldots, P_r(T_1, \dots, T_n, Y) \in k[T_1,\dots,T_n][Y]$ of group $G$ over $k(T_1, \dots, T_n)$ fulfilling this: for every $E/k \in {\sf{R}}_G(k)$, there are $i \in \{1, \dots, r\}$ and $(t_1,\ldots,t_n) \in k^n$ such that $E$ is the splitting field over $k$ of $P_i(t_1,\ldots,t_n,Y)$ (see \cite[Theorems 3.3 and 4.2]{MS93})\footnote{By the proof of \cite[Theorem 4.2]{MS93}, one can actually take $n=5$.}. Such a set is called a {\it finite $k$-parametric set of polynomials for $G$}. However, there exists no $k$-parametric polynomial $P(T_1, \dots, T_n, Y) \in k[T_1,\ldots,T_n][Y]$ of group $G$ over $k(T_1, \dots, T_n)$, for any number of variables $n$ (see \cite[Remark A.2]{KN20}). 
\end{remark}

Our second remark suggests a notion of ``parametric dimension" measuring the complexity of all the Galois extensions of a given field with any given finite Galois group.

\begin{remark}
Recall that the {\it{generic dimension}} of a finite group $G$ over a field $k$, de\-no\-ted by ${\rm{gd}}_k G$, is either the smallest $n \geq 1$ such that there is a generic polynomial $P(T_1, \dots, T_n, Y) \in k[T_1, \dots, T_n][Y]$ of group $G$, or $\infty$ if there is no generic polynomial of group $G$ with coefficients in $k$ (see \cite[\S8.5]{JLY02}). In view of Remark \ref{rem:finite}, for the analogous notion of parametric dimension, we allow finite $k$-parametric sets of polyno\-mials. Define the {\it{(generalized) parametric dimension}} of $G$ over $k$, denoted by ${\rm{pd}}_k G$, to be either the smallest $n \geq 1$ for which there exists a finite $k$-parametric set of polynomials in $k[T_1, \dots, T_n][Y]$ for $G$, or $\infty$ if there is no such set\footnote{Note that the splitting fields of the given polynomials have Galois group $G$ over the rational function field $k(T_1,\ldots,T_n)$, as opposed to the definition of ``essential parametric dimension" in \cite{KN20}, which allows extensions of arbitrary function fields.}. 

Clearly, ${\rm{pd}}_k G \leq {\rm{gd}}_k G$, and it may happen that equality does not hold. For example, if $k$ is PAC (the definition is recalled in \S\ref{sec:intro}), we always have ${\rm{pd}}_k G=1$ (while ${\rm{gd}}_k G \geq 2$ for many groups $G$; see \cite[Proposition 8.2.4]{JLY02} and \cite{CHKZ08}), by \cite{Deb99a} and the fact that the answer to the regular inverse Galois problem over PAC fields is positive (see \cite[Main Theorem A]{Pop96}). A family of non-PAC examples is given by Remark \ref{rk:PAC}: if $k$ is algebraically closed of characteristic 0, then we always have ${\rm{pd}}_{k((V))} G=1$.
\end{remark}

Finally, let us recall the following definition (see \cite{BR97, JLY02}):

\begin{definition}
Let $k$ be a field.

\vspace{0.5mm}

\noindent
{\rm{(1)}} Let $M/L$ be a finite separable field extension with $k \subseteq L$. If, for an intermediate field $k \subseteq L' \subseteq L$, there is a field extension $M'$ of $L'$ contained in $M$, and with $[M':L']=[M:L]$ and $M=M'L$, we say that $M/L$ is {\it{defined}} over $L'$. Moreover, the {\it{essential dimension}} of $M/L$ over $k$ is the minimum of the transcendence degree of $L'/k$, when $L'$ runs through all intermediate fields over which $M/L$ is defined.

\vspace{0.5mm}

\noindent
{\rm{(2)}} Let $G$ be a finite group, acting regularly on a set ${\bf{T}} = \{T_1, \dots, T_{|G|}\}$ of indeterminates. Then the {\it{essential dimension}} ${\rm{ed}}_k G$ of $G$ over $k$ is the essential dimension of $k({\bf{T}})/k({\bf{T}})^G$ over $k$.
\end{definition}

Note that ${\rm{ed}}_k G \leq {\rm{gd}}_k G$ and, when ${\rm{gd}}_k G$ is finite, it is conjectured that ${\rm{ed}}_k G = {\rm{gd}}_k G$ (see \cite[\S 8.5]{JLY02}). However, the following known example (see \cite[Theorem 3.4]{One12}) shows that ${\rm{pd}}_k G$ may be strictly smaller even than ${\rm{ed}}_k G$ (with $k$ a number field).

\begin{example} 
Let $k=\mathbb Q(\sqrt{-1})$, $G=(\mathbb Z/2\mathbb Z)^5$, and consider the three polynomials $P_i(T_1,T_2, T_3,T_4,Y) = (Y^2-T_1-\ldots-T_i)\prod_{i=1}^4(Y^2-T_i)$, $i=2,3,4$, over $k(T_1, T_2, T_3,T_4)$. We claim that $\{P_2, P_3,P_4\}$ is a finite $k$-parametric set for $G$, and hence ${\rm{pd}}_k G \leq 4 <5={\rm{ed}}_k G$ (see, e.g., \cite[Theorem 8.2.11]{JLY02} for the last equality).

To show the claim, note that any Galois extension $E$ of $k$ of group $G$ is of the form $E=k(\sqrt{t_1},\ldots,\sqrt{t_5})$ for some $t_i\in k^\times$, $i=1,\ldots,5$. Then, by the Hasse--Minkowski theorem (see \cite[Chapter VI, Corollary 3.5]{Lam05}), and as $k$ is totally imaginary, there is $(a_1, \dots, a_4) \in k^4$ such that $t_5=\sum_{j=1}^4 t_j a_j^2$ (up to reordering the $t_i$'s). Rearrange the values of $t_j$, so that $a_j\neq 0$ for $j\leq i$, and $a_j=0$ for $j>i$, for some $2\leq i\leq 4$ (note that $i=1$ is impossible as $[k(\sqrt{t_1}, \sqrt{t_5}) : k]=4$). The splitting field of $P_i(s_1,\ldots,s_4,Y)$ over $k$, where $s_j = t_j a_j^2$ for $j \leq i$ and $s_j = t_j$ for $j > i$, is then $E$, as desired. 
\end{example}

Ongoing research will investigate further the connection between generic, essential, and parametric dimensions.

\vspace{-1.5mm}

\bibliography{Biblio2}
\bibliographystyle{alpha}

\end{document}